\documentclass[11pt]{article}

\usepackage{authblk}
\usepackage{amsmath,amsthm,amsfonts,amssymb,mathrsfs,epsfig,bm}

\usepackage[usenames]{color}
\usepackage[colorlinks=true]{hyperref}
\usepackage{graphicx}
\usepackage{marginnote}
\parskip        \medskipamount
\newtheorem{theorem}{Theorem}%[section]
\newtheorem{lemma}{Lemma}
\newtheorem{corollary}{Corollary}
\newtheorem{proposition}{Proposition}
\newtheorem{definition}{Definition}

\newtheorem{remark}{Remark}
\theoremstyle{remark}

\def\R{\mathbb{R}}

\def\P{\mathbb{P}}
\def\E{\mathbb{E}}
\def\I{\mathcal{I}}
\def\K{\mathcal{K}}

\def\G{\mathcal{G}}
\def\F{\mathcal{F}}
\def\M{\mathcal{M}}
\def\W{\mathcal{W}}

\def\SS{\mathscr{S}}

\def\LL{\mathcal{L}}

\def\gr{\operatorname{gr}}
\renewcommand{\phi}{\varphi}
\renewcommand{\epsilon}{\varepsilon}

\definecolor{mygray}{gray}{0.9}
\definecolor{deeppink}{RGB}{255,20,147}
\definecolor{mygreen}{rgb}{0.05, 0.576, 0.03}
\definecolor{myred}{rgb}{0.768, 0.09, 0.09}

\long\def\symbolfootnote[#1]#2{\begingroup
\def\thefootnote{\fnsymbol{footnote}}\footnote[#1]{#2}\endgroup}
\newcommand{\Keywords}[1]{ \noindent {\footnotesize
             {\small \em Keywords and phrases.} {\sc #1} } }
\newcommand{\ams}[2]{  \noindent {\footnotesize
             {\small \em AMS {\rm 2000} subject classifications.
           {\rm Primary {\sc #1}; secondary {\sc #2}} } } }

\begin{document}

\title{Kingman's model with
 random mutation probabilities: \\convergence and condensation I}
\author{Linglong Yuan\footnote{University of Liverpool,
Department of Mathematical Sciences, Peach Street, L69 7ZL, Liverpool, UK. \\Email: yuanlinglongcn@gmail.com}\,\,\footnote{
Xi'an Jiaotong-Liverpool University,
Department of Mathematical Sciences, Ren'ai Road, 215111, Suzhou, P.R.C.\ China.}}

\affil[]{Department of Mathematical Sciences\\
University of Liverpool}
\affil[]{Department of Mathematical Sciences\\
Xi'an Jiaotong-Liverpool University}
%\email{yuanlinglongcn@gmail.com}
%\affil[]{Department of Mathematics\\ Johannes-Gutenberg-Universit\"at \ln}
%\email{yuanlinglongcn@gmail.com}
%\affil[]{Department of Mathematics\\ Uppsala University}
%\Date{October 2014}
\date{\today}

\maketitle
\tableofcontents

\begin{abstract}
For a one-locus haploid infinite population with discrete generations, the celebrated Kingman's model describes the evolution of fitness distributions  under the competition of selection and mutation, with a constant mutation probability. This paper generalises Kingman's model by using i.i.d.\ random mutation probabilities, to reflect the influence of a random environment. The weak convergence of fitness distributions to the globally stable equilibrium is proved. The condensation occurs when almost surely a positive proportion of the population travels to and condensates at the largest fitness value. The condensation may occur when selection is more favoured than mutation. A criterion for the occurrence of condensation is given. 
\end{abstract}

 % insert keywords separated by a semicolon

\ams{60F05}{60G10, 92D15, 	92D25} % insert the primary Maths Subject Classification number in the first bracket
         % and the secondary ams number(s) in the second bracket
         % e.g. \ams{60E20}{49G03;49F10}

\section{Motivation and background}

Various biological forces interact with each other and drive the evolution of population all together. 
One important competing pair consists of selection and mutation. It was as early as 1937 when Haldane \cite{H37} put forward the concept  of  mutation-selection balance. The mathematical foundation of this subject was established by 
Crow and Kimura \cite{CK70}, Ewens \cite{E79}, and Kingman \cite{K80}. For more details on this topic, we refer to B\"urger \cite{B98, B00}. 

A simple setting is to consider a one-locus haploid infinite population with discrete generations under selection and mutation. The locus is assumed to have infinitely many possible alleles which have continuous effects on a quantitative type. The continuum-of-alleles models were introduced by Crow and Kimura \cite{CK70} and Kimura \cite{K65} and are used frequently in quantitative genetics. 

Kingman \cite{K77} suggested to explain the tendency that most mutations are deleterious by the assumption of the independence of the gene before and after mutation. This feature was named ``House of Cards",  as the mutation destroys the biochemical house of cards built up by evolution, by Kingman in \cite{K78} where the most famous one-locus model was proposed.  In this model, a population is characterised by its type distribution, which is a probability measure on $[0,1]$ and any $x\in[0,1]$ is a type value. In Kingman's setting, an individual with a larger type value is fitter, which means more productive. So the type value can also be named fitness value.
Kingman's model can be seen as the limit of a finite population model, see \cite{G19}.

B\"urger \cite{B89} generalised the selection mechanism by allowing the gene after mutation to depend on that before and proved convergence in total variation. The genetic variation of the equilibrium distribution was computed and discussed.  I proposed \cite{Y15} a more general selection mechanism which can model general macroscopic epistasis, with the other settings the same as in Kingman's model. This model was applied to the modelling of the Lenski experiment (see \cite{GKWY} for a description of the experiment).

There exist also models on the balance of selection and mutation in the setting of continuous generations. B\"urger \cite{B86} provided an exact mathematical analysis of Kimura's continuum-of-alleles model, focusing on the equilibrium genetic variation. Steinsaltz {\it et al} \cite{SEW05} proposed a multi-loci model using a differential equation to study the ageing effect. Later on the recombination 
was incorporated to the model \cite{SEW13}. Betz {\it et al}'s model \cite{BDM18} generalised a continuous-time version of Kingman's model and other models arising from physics. 

However to the author's best knowledge, Kingman's model has never been generalised to a random version. 
In this paper we will assume that the mutation probabilities of all generations form an i.i.d.\ sequence. Biologically, we think of a stable random environment such that the mutation probabilities vary on time but  independently sampled from the same distribution. 

In Kingman's model, condensation occurs if a certain positive proportion of the population travels to and condensates at the largest fitness value. This is due to the dominance of selection over mutation. In the random model proposed in this paper, we also consider the convergence of (random) fitness distributions to the equilibrium and the condensation phenomenon. Moreover, Kingman's model has been revisited recently in terms of the travelling wave of mass to the largest fitness value \cite{DM13}. The random model provides another example for consideration in this direction.

\section{Models}

\subsection{Kingman's model with time-varying mutation probabilities}

Consider a haploid population of infinite size and discrete generations under the competition of selection and mutation. We use a sequence of probability measures $(P_n)=(P_n)_{n\geq 0}$ on $[0,1]$ to describe the distribution of fitness values in the $n$th generation. We can assume, more generally, that the probability measures are supported on a finite interval, not necessarily $[0,1]$. But since only fitness ratios will be  relevant (see \cite{K78} or \cite{Y15} for a more explicit explanation), we adopt the setting of $[0,1]$, which was used by Kingman \cite{K78}, and which is equivalent to general finite supports. 

Individuals in the $n$th generation are 
children of the $(n-1)$th generation. First of all, the fitness distribution of children is initially $P_{n-1}$ (an exact copy from parents). Then selection takes effect, such that the fitness distribution is updated from $P_{n-1}$ to the size-biased distribution
$$\frac{x P_{n-1}(dx)}{\int y P_{n-1}(dy)}.$$ 
Here we use $\int$ to denote $\int_0^1.$ Basically the new population is re-sampled from the existing population by using their fitness as a selective criterion. 
Next, each individual mutates independent with the same mutation probability which we denote by $b_{n}$ taking values in $[0,1).$ Each mutant has the fitness value sampled independently from a common mutant distribution, that we denote by $Q$, a probability measure on $[0,1]$. Then the resulting distribution is the distribution of the $n$th generation
\begin{equation}\label{pi} P_{n}(dx)=(1-b_{n})\frac{x P_{n-1}(dx)}{\int y P_{n-1}(dy)}+b_{n}Q(dx).\end{equation}
The fact that we exclude the case that $b_n$ equals $1$ is because in this situation  we have $P_{n}=Q$ which loses accumulated evolutionary changes. This is not interesting neither biologically nor mathematically.  

Expanding (\ref{pi}), we can also obtain
\begin{equation}\label{P'}P_{n}(dx)=\left(\prod_{l=0}^{n-1}\frac{1-b_{l+1}}{\int yP_l(dy)}\right)x^{n} P_0(dx)+\sum_{j=1}^{n}\left(\prod_{l=j}^{n-1}\frac{1-b_{l+1}}{\int yP_l(dy)}\right)b_jm_{n-j}Q^{n-j}(dx)\end{equation}
where 
$$Q^k(dx):=\frac{x^kQ(dx)}{\int y^kQ(dy)},\quad m_k:=\int x^kQ(dx),  \quad \forall k\geq 0.$$
In particular if $Q=\delta_0,$ the Dirac measure on $\{0\}$,  $Q^k=\delta_0$ for any $k\geq 0$.

When all the $b_n$'s are equal to the same number $b\in[0,1),$ this is the model introduced by Kingman \cite{K78}. In the general setting we allow the mutation probabilities to be different. We call it {\it Kingman's model with time-varying mutation probabilities} or {\it the general model} for short.  %The term ``House-of-Cards'' was firstly used by Kingman \cite{K78}, since the mutation destroys the biochemical ``house of cards'' built up by evolution, as the fitness values of the mutants are independent of that of the parents. 

We introduce a few more notations. Let $M$ be the space of (nonnegative) Borel measures on $[0,1]$ and $M_1$ the subspace of $M$ consisting of probability measures. Let $M, M_1$ be endowed with the topology of weak convergence. We use $\stackrel{d}{\longrightarrow}$ to denote weak convergence.  We say a sequence of measures $(u_n)$ converges in total variation to a measure $u$, denoted by $u_n\stackrel{TV}{\longrightarrow}u,$ if the total variation, $\sup_B|u_n(B)-u(B)|$ where the supremum is taken over all Borel sets, converges to $0$.

For any $u\in M_1$, define
\begin{equation}\label{u}S_{u}:=\sup\{x:u[x,1]>0\}.\end{equation}
So $S_{u}$ is interpreted as the {\it largest fitness value} in a population of distribution $u.$ Define $h:=S_{P_0}$. It is not difficult to see that $S_{P_{n}}=\max\{S_{P_0}, S_Q\}$ if the equality holds for $n-1$ or $0<b_n< 1$. Since we are interested in asymptotics, it is thus without loss of generality to assume that $h\geq S_Q$. Therefore $S_Q\leq h\leq 1.$

Note that the general model has parameters $(b_n)_{n\geq 1}, Q, P_0, h$.  Kingman's model shares the same parameters, but with $b_n$'s all equal to $b.$
 We call $(P_n)$ {\it the forward sequence} or just {\it the sequence}. Although $h$ is determined by $P_0$, we still consider $h$ as a parameter as it will be clear later that for Kingman's model and the random model considered in this paper, the limit of $(P_n)$ depends on $P_0$ only through $h$. This is the so-called global stability.  

\subsection{Convergence and condensation in Kingman's model}
Kingman \cite{K78} proved the convergence of $(P_n)$ when all mutation probabilities are equal, i.e., $b_n=b, \forall n\geq 1$.

\begin{theorem}[Kingman's Theorem, \cite{K78}]\label{King}
1.  If $\int \frac{Q(dx)}{1-x/h}\geq b^{-1},$ then $(P_n)$ converges in total variation  to 
$$\K(dx)=\frac{b \theta_bQ(dx)}{\theta_b-(1-b)x},$$
with $\theta_b$, as a function of $b$, being the unique solution of 
\begin{equation}\label{sb}\int\frac{b \theta_bQ(dx)}{\theta_b-(1-b)x}=1.\end{equation}
2. If $\int \frac{Q(dx)}{1-x/h}< b^{-1}$, then $(P_n)$ converges weakly to 
$$\K(dx)=\frac{b Q(dx)}{1-x/h}+\Big(1-\int\frac{b Q(dy)}{1-y/h}\Big)\delta_{h}(dx).$$ 
%here $\delta_h(dx)$ is the Dirac measure at $h.$ 
\end{theorem}
Note that $\K$ is uniquely determined by $b, Q, h$, but not the choice of $P_0$. In this sense $\K$ is a globally stable equilibrium. For simplicity, for any measure, say $\mu,$ its mass on a point $x$ is denoted by $\mu(x)$ instead of $\mu(\{x\}).$ Then we say {\it there is  condensation at $h$ in Kingman's model} if $Q(h)=0$ but $\K(h)>0$. We call $\K(h)$ {\it the condensate size} if $Q(h)=0$.
In the case 1 above, there is no condensation. 
The condition $\int \frac{Q(dx)}{1-x/h}\geq b^{-1}$ is satisfied only if $b$ is big and/or $Q$ is fit (i.e., having more mass on larger values). It means mutation is stronger against selection, so that the limit does not depend on $P_0$ at all.

In the case 2, the condition $\int \frac{Q(dx)}{1-x/h}< b^{-1}$ implies $Q(h)=0$, but we have that $\K(h)>0$.  So there is condensation. Contrarily to the first case, selection is more favoured so that the limit depends on $P_0$ through $h$. If $P_0(h)=0$ (implying $S_{P_n}=h$ and $P_n(h)=0$ for any $n$), a certain amount of mass $\K(h)=\Big(1-\int\frac{b Q(dy)}{1-y/h}\Big)$ travels to the largest fitness value $h$, by the force of selection. 

Next we introduce the random model, which is the main object of study in this paper. 
\subsection{Kingman's model with random mutation probabilities}

%Let $\LL\in M_1$ and 
Let $(\beta_n)_{n\geq 0}$ be an i.i.d.\ sequence of random variables in the common probability space $(\Omega, \F, \P)$, taking values in $[0,1)$ with common distribution $\LL\in M_1$ supported on $[0,1)$. 
{The {\it Kingman's model with random mutation probabilities} or simply {\it the random model} is defined by the following dynamical system: 
\begin{equation}\label{ipi} P_{n}(dx)=(1-\beta_{n})\frac{x P_{n-1}(dx)}{\int y P_{n-1}(dy)}+\beta_{n}Q(dx), \quad n\geq1.\end{equation}
The random model has parameters $(\beta_n), Q, P_0, h$. It is a randomisation of Kingman's model, as we can set $\beta_n$'s to equal $b$ with probability 1. 

We are interested in the convergence of $(P_n)$ to the equilibrium and the phenomenon of condensation. 
Since we are dealing with random probability measures, i.e., random elements of $M_1,$  let us recall the definition of weak convergence in this context. Random (probability) measures $(\mu_n)$ supported on $[0,1]$ converge weakly to a limit $\mu$ if and only if for any continuous function $f$ on $[0,1]$ we have 
$$\int f(x)\mu_n(dx)\stackrel{d}{\longrightarrow}\int f(x)\mu(dx).$$
We refer to \cite{K17} for a reference on random measures. The definition of weak convergence for random measures stated in the follow-up paper \cite{Y20} is incorrect. But it does not affect anything there as the weak convergence results are all proved in this paper. 
%This definition applies to the weak convergence of deterministic (probability) measures as well, by taking away the expectation.

As the sequence $(P_n)$ is completely determined by $(\beta_n), Q, P_0$ and $h$, the only randomness arises from $(\beta_n)$. In comparison to the terminology in statistical physics, the weak limit of $(P_n)$ is an {\it annealed} limit, which is obtained given the law of $(\beta_n)$. A {\it quenched} limit, which is obtained by conditioning on $(\beta_n)$, does not exist unless $P_0=Q=\delta_0$. A simple reason for nonexistence is that $P_n$ contains $\beta_nQ$ which fluctuates persistently as $(\beta_n)$ is i.i.d..  However in Section \ref{fbs} we will see that it is possible to obtain a quenched limit if the 
evolution is seen backwards.

For the particular case that $Q=\delta_0$, we have
$$P_n(dx)=(1-\beta_{n})\frac{x^nP_0(dx)}{\int y^nP_0(dy)}+\beta_{n}\delta_0(dx).$$
From this, it is easily deduced that the sequence $(P_n)$ converges weakly to the random element $(1-\beta)\delta_h+\beta\delta_0$, where $\beta$ is a random variable with law $\LL$, the common law of the $\beta_n$'s. So we assume from now on
$Q\neq\delta_0.$

%So for $Q=\delta_0$, the weak convergence always holds and condensation always occurs.  

\section{Main results}

%The existence of invariance measures can be shown by introducing {\it a finite back sequence}. 
\subsection{Weak convergence}%{\bf \noindent 1). Weak convergence.} 
Recall that the sequence $(P_n)$ in the random model has parameters $(\beta_n), Q, P_0$ and $h$, with $h=S_{P_0}.$ Then $(P_n)$ converges weakly to a globally stable equilibrium, in the sense that the limit depends on $P_0$ only through $h$. Recall $\beta$ is a random variable with law $\LL$, the common law of $\beta_n$'s.

%Since $\beta_n$ is independent of $\I_n$, and $\I_{n-1}\stackrel{d}{=}\I_n$ for any $j\geq1$, we conclude that $\I_0$ is an invariant measure of the i-i-d model.  

%It will be more clear later that $\I_0$
%plays an important role in the analysis, so we give a new notation to emphasize its importance
%$$\I_*^h=\I_0.$$

\begin{theorem}\label{main}
For the random model \eqref{ipi}, the sequence 
$(P_n)$ converges weakly to a random probability measure, denoted by $\I$, whose distribution depends on $\LL, Q, h$ but not on the choice of $P_0$.
\end{theorem}
\begin{remark}
In (\cite{Y20}, P.872),  it is written that the distribution of $\I$ depends on $\beta, Q, h$. The statement is true in the sense that the distribution of $\I$ depends on $\beta$ via its distribution. Here we make it more clear by replacing $\beta$ by $\LL$. 
\end{remark}
\begin{remark}\label{0sq}
If we start with $P_0=\delta_h$ (recall that $h\in[S_Q,1]$), then all $P_n$'s are supported on $[0,S_Q]\cup\{h\}$, which entails that the limit $\I$ is supported on the same set $[0,S_Q]\cup\{h\}$. 
Moreover we have either $\I(h)>0$ a.s. or $\I(h)=0$ a.s. (a justification is provided in Remark \ref{eitheror} in Section \ref{lfbs}).
In the latter case, $\I$ is supported only on $[0,S_Q]$ and (the distribution of) $\I$ does not depend on $h$ (see Theorem \ref{critical}). 
Therefore, although we say $h$ is a parameter of $\I$ but it should be understood in the sense that $\I$ is weak limit of $(P_n)$ with $h=S_{P_0}.$
\end{remark}
The limit $\I$ is introduced in Section \ref{fbs}. But the proof of weak convergence is deferred to a later stage, as it uses other main results such as the condensation criterion for the random model. 

\subsection{Condensation criterion}%{\bf \noindent 2). Condensation criterion.} 
 %It is clear that if $h=S_Q$ and $Q(S_Q)>0$, we have $\I_Q(h)>0, a.s.$. 
 The fact that either $\I(h)>0$ a.s. or $\I(h)=0$ a.s. allows us to give the precise definition of condensation in line with that for Kingman's model, as follows:
\begin{definition}
For the random model, we say there is condensation at the largest fitness value $h$ if $Q$ assigns zero mass at $h$ (i.e., $Q(h)=0$) but the limiting measure $\I$ assigns positive mass at it (i.e., $\I(h)>0$, a.s.). 
\end{definition}

\noindent Next we give the condensation criterion. If $h=S_Q$, we write $\I_Q$ for $\I$ and $\K_Q$ for $\K$.

\begin{theorem}[Condensation criterion]\label{critical}
If there is no condensation at $h$, then $\I\stackrel{d}{=}\I_Q$. The condensation criterion for $\I$ at $h$ is as follows: 
\begin{enumerate}
\item If $h=S_Q$, then there is no condensation at $h$ if 
\begin{equation}\label{sq}\E\left[\ln \frac{S_Q(1-\beta)}{\int y\I_Q(dy)}\right]<0.\end{equation}
\item If $h>S_Q$, then there is no condensation at $h$ if and only if 
\begin{equation}\label{h}\E\left[\ln \frac{h(1-\beta)}{\int y\I_Q(dy)}\right]\leq 0.\end{equation}
\end{enumerate}
Here $\E\left[\ln \frac{1-\beta}{\int y\I_Q(dy)}\right]$ is well defined, and takes values in $[-\infty,-\ln\int yQ(dy)]$, and depends only on the marginal distributions of $\beta$ and $\I_Q.$ 
\end{theorem}
\begin{remark}\label{i=iq}
In fact, if there is no condensation at $h$, then $\I,\I_Q$ are the same random probability measure, based on the definition of $\I$ introduced at the end of Section \ref{fbs}. But since here we do not have the definition yet, we write a weaker version $\I\stackrel{d}{=}\I_Q.$ 
\end{remark}
By Remark \ref{e<=0} in Section \ref{lfbs}, we can only have 
$\E\left[\ln \frac{S_Q(1-\beta)}{\int y\I_Q(dy)}\right]\leq0$. About the occurrence of condensation in the case where $h=S_Q$, the fact that we cannot say anything when $\E\left[\ln \frac{S_Q(1-\beta)}{\int y\I_Q(dy)}\right]=0$ can be better understood in Kingman's model, which is a special random model.  In this model, $\E\left[\ln \frac{S_Q(1-\beta)}{\int y\I_Q(dy)}\right]=0$ becomes $$\ln \frac{S_Q(1-b)}{\int y\K_Q(dy)}=0.$$
By some simple computations using Theorem \ref{King}, the above display is equivalent to 
$$\int\frac{Q(dx)}{1-x/S_Q}\leq b^{-1}.$$ But it covers cases with and without condensation. 
For full details please see Appendix \ref{=<0}, where the case $h>S_Q$ is also analysed. 

We give some intuition why Theorem \ref{critical} holds true.  Consider the unnormalised variant of the dynamical system that is given by 
\begin{equation}\label{*}\overline{P}_n(dx)=(1-\beta_n)x\overline{P}_{n-1}(dx)+\beta_n\left(\int y\overline{P}_{n-1}(dy)\right)Q(dx)\end{equation}
with $\overline{P}_0=P_0.$ By induction, it can be shown  that 
\begin{equation}\label{**}\overline{P}_n=P_n\prod_{i=0}^{n-1}\int xP_i(dx), \forall n\geq 0.\end{equation}
We can roughly think of the growth of $\overline{P}_n$ as contributed by two parts, the initial $P_0$ and the subsequently arriving $Q$'s. If the initial distribution is supported on $[0,S_Q]$,  by Theorem \ref{main}, $P_i$ converges weakly to $\I_Q$ as $i\to\infty.$ Then the part of $\overline{P}_n$ contributed by the $Q$'s grows at rate $\gr(Q):=\E[\ln\int x\I_Q]$ (see \eqref{**}). In comparison the largest fitness value $h$ in $P_0$ can be assigned the growth rate $\gr(h):=\E[\ln h(1-\beta)]$ (due to the term $(1-\beta_n)x\overline{P}_{n-1}(dx)$ in \eqref{*}). 
Then it is clear that the occurrence of condensation is determined by the comparison of $\gr(h)$ and $\gr(Q)$. However it is subtle when $\gr(h)=\gr(Q)$: no condensation if $h>S_Q$ and it is undetermined if $h=S_Q$.

In the follow-up paper \cite{Y20}, we provide a matrix representation for $\I_Q$, so the condensation criterion can be written neatly (Corollary 2, P.877). Moreover, using matrix analysis, we can compare the fitness of equilibria from different models (Section 3.3-(3), P.878--879).  The challenging problem of finding a necessary and sufficient condition for the occurrence of condensation in the case $h=S_Q$ has not been dealt with anywhere and still remains open.   

\subsection{Invariant measure}%{\bf \noindent 3). Invariant measure.} 
We introduce the notion of {\it invariant measure}, which includes the limit $\I.$ We will heavily use the invariant measures in the proofs.
\begin{definition}[Invariant measure]
A random probability measure $\nu$ is invariant if it is supported on $[0,1]$ and satisfies\ 
\begin{equation}\label{invnu}\nu(dx)\stackrel{d}{=}(1-\beta)\frac{x\nu(dx)}{\int y\nu(dy)}+\beta Q(dx)\end{equation} 
where $\beta \text{ is independent of } \nu$.
\end{definition}
Clearly $\I$ is an invariant measure, since it is the weak limit of $(P_n)$ defined by \eqref{ipi}.

\begin{theorem}[Compoundness of invariant measures]\label{com}
For any invariant measure $\nu$, there exists a regular conditional distribution of $\nu$ on $S_\nu$. Moreover, conditional on $S_\nu$,  
$$(\nu|S_{\nu})\stackrel{d}{=}\I,\quad \text{ almost surely},$$
where $\I$ is the random probability measure introduced in Theorem \ref{main} with parameters  $\LL, Q, h=S_\nu$ and satisfies $\P(S_\I=S_\nu|S_\nu)=1$, a.s..
\end{theorem}
\begin{remark}
Remark \ref{0sq} says that if there is no condensation at $h$, then $\I$ is supported on $[0,S_Q]$. Since $\I$ is an invariant measure, the above theorem entails that $\I\stackrel{d}{=}\I_Q.$ This assertion has been stated in Theorem \ref{critical}. %Another issue is that the well-definedness of $\I_\nu$ has to be and will be justified in the proof. 
\end{remark}

Using the notion of invariant measures, we can solve a distributional equation in the following example.  
For a survey on distributional equations, we refer to Aldous and  Bandyopadhyay \cite{AB05}. 

\vspace{2mm}

\noindent{\bf Example 1}. Consider a particular case: $Q$ is supported only on $\{c\}$ for some $c\in (0,1)$, and $h\in (c,1)$. Let $\nu$ be an invariant measure supported on $\{c\}\cup\{h\}$. Then $\nu$ can be written as $\nu=X\delta_c+(1-X)\delta_h$ where $X$ is a 
random variable taking values in $[0,1]$, and satisfies 
$$X\delta_c+(1-X)\delta_h\stackrel{d}{=}(1-\beta)\frac{cX\delta_c+h(1-X)\delta_h}{cX+h(1-X)}+\beta\delta_c,$$
where $\beta$ is independent of $X.$ The above display is equivalent to  
$$X\stackrel{d}{=}\frac{c+(h\beta-c)(1-X)}{c+(h-c)(1-X)}.$$
We are interested in a necessary and sufficient condition for the above equation to have a solution $X$ with $0\leq X<1$ a.s.\ (i.e., $\nu(h)>0$ a.s.). By Theorem \ref{com}, it is equivalently saying that there is condensation at $h$. 
By Theorem \ref{critical},  the necessary and sufficient condition is simply $\E[\ln (h(1-\beta)/c)]>0.$ Moreover as such $\nu$ is unique (in distribution), the solution $X$ is also unique (in distribution). 

The paper is organised as follows. Section \ref{relameas} and \ref{three} provide necessary preparations. Section \ref{fbs} and \ref{lfbs} analyse the finite backward sequence, which is the main tool used in this paper.  Section \ref{proofcritical} proves Theorem \ref{critical}. Section \ref{someinv} analyses the invariant measures, and the results obtained there will be used in Section \ref{proofmain} to prove the weak convergence in Theorem \ref{main}. Section \ref{proofcom} is dedicated to the proof of Theorem \ref{com}.

\section{Proofs}
\subsection{Relations between measures}\label{relameas}
We introduce some notations to describe relations between measures.

\noindent 1). For measures $u,v\in M$, we say $u$ is a component 
of $v$ on $[0,a]$ (resp. $[0,a)$), denoted by $u\leq_a v$ (resp. $\leq_{a-}$), if 
$$u(A)\leq v(A), \quad\text{ for any measurable set } A\subset [0,a] \,\,(\text{resp. } [0,a)).$$
%Write $\leq$ when $a=1$. 

For random measures $\mu,\nu\in M$, we write 
$\mu\leq_a^d \nu$ if there exists a coupling $(\mu',\nu')$ with $\mu',\nu'\in M$ such that 
\begin{equation}\label{precran}\mu'\leq_a \nu'  \text{  a.s. and }\,\, \mu'\stackrel{d}{=}\mu,\,\,\nu'\stackrel{d}{=}\nu.\end{equation}
The relation $\mu\leq_{a-}^d \nu$ is defined in a similar way. 

\noindent 2). For measures $(u_n)$  and $u$ in $M$, we introduce a notation
$$u_n\leq_a\stackrel{TV}{\longrightarrow}u$$
which means that $u_n\leq_a u_{n+1}$ for any $n$, and $u_n$ converges in total variation to $u.$
We define similarly $\leq_{a-}\stackrel{TV}{\longrightarrow}.$ %We also define $\leq_{a}^d\stackrel{TV}{\longrightarrow}$ and $\leq_{a-}^d\stackrel{TV}{\longrightarrow}$ for random measures as in (\ref{precran}). 

\noindent 3). For real-valued random variables $\xi, \eta,$ we write the well known {\it stochastic ordering} $\xi\preceq\eta$, which holds if 
$$ \P(\xi\leq x)\geq \P(\eta\leq x),\quad \forall\, x\in\R.$$

\noindent 4). For any $u\in M_1,$ let the distribution function of $u$ be $$D_{u}(x):=u([0,x]),\quad \forall x\in[0,1].$$ 
For any $u, v\in M_1,$ we use the same notation $\preceq$ of stochastic ordering and write $u\preceq v$ if $D_u(x)\geq D_v(x)$ for any $x\in[0,1]$. This definition is natural, as $\xi\preceq \eta$ is equivalent to $u\preceq v$, if $u$ is the distribution of $\xi$ and $v$ is the distribution of $\eta.$ %For random probability measures $\mu,\nu$, we also define $\mu\preceq^d \nu$ similarly as in (\ref{precran}). 

\begin{remark}\label{meascomp}We make a comment between $\leq_{a-}$ and $\preceq.$ For two probability measures $u,v\in M_1$, assume that $S_u=S_v=a$, then $u\leq_{a-}v$ implies that $v\preceq u.$ But the converse is not true. 
 \end{remark}

\begin{remark}
If we use notations 
$\leq_{a}, \leq_{a-}, \leq_{a}\stackrel{TV}{\longrightarrow}, \leq_{a-}\stackrel{TV}{\longrightarrow}, \preceq$  to describe the relations between random measures, it should be understood that they hold in the almost sure sense, or even pointwise sense (i.e., for every $\omega\in\Omega$) if possible.   

Similarly if we use $\leq, <,\geq,>,=,\neq$ to compare random variables, it should be understood in the almost sure sense, or pointwise sense. 
\end{remark}
%For any two measures $\mu,\nu\in M$, we say $\mu$ is 
%{s%tochastically dominated} by $\nu$, denoted by $\mu\leq \nu$,  if
%$$D_\mu(x)\geq D_\nu(x), \quad\forall \,x\in[0,1].$$

%Similarly, 

\subsection{Three sequences}\label{three}
To study the asymptotic behaviour of $(P_n)$, we also introduce $(P_n'), (P_n'')$ so that the three forward sequences correspond respectively to 
$$((\beta_n), Q, P_0, h),\quad  ((\beta_n), Q', P_0', h'),\quad ((\beta_n), Q'', P_0'', h'').$$
The parameters of $(P_n')$ and $(P_n'')$ will be specified when they are used. The two sequences will converge weakly when they are used, and $(P_n)$ is compared to them or one of them to show that $(P_n)$ also converges weakly. The first place where this technique is used is in Section \ref{someinv}. 

  Using (\ref{P'}), we write
\begin{equation}\label{hg}
P_n(dx)=\M_n(dx)+\W_n(dx)
\end{equation}
with 
$$\M_n(dx)=\Big(\prod_{l=0}^{n-1}\frac{1-\beta_{l+1}}{\int yP_l(dy)}\Big)x^{n} P_0(dx) $$
and 
$$\W_n(dx)=\sum_{j=1}^{n}\Big(\prod_{l=j}^{n-1}\frac{1-\beta_{l+1}}{\int yP_l(dy)}\Big)b_jm_{n-j}Q^{n-j}(dx).$$
Therefore $\M_n$ is the contribution to $P_n$ made by $P_0$, while $\W_n$ is the contribution by the $Q$'s. 

Similarly we introduce
\begin{equation}\label{hg'}
P_n'(dx)=\M_n'(dx)+\W_n'(dx)
\end{equation}
\begin{equation}\label{hg''}
P_n''(dx)=\M_n''(dx)+\W_n''(dx)
\end{equation}
with $\M_n', \W_n', \M_n'', \W_n''$ defined correspondingly. 

\subsection{Introducing the finite backward sequences}\label{fbs}
%In this part we show the convergence of $(P_n)$ in the random model. But to explain what the limit is, we need some notations and some small results.
 \subsubsection{The general model.}
We introduce the finite backward sequence $( P_j^n)=( P_j^n)_{0\leq j\leq n}$ for the general model which has parameters $n, (b_j)_{1\leq j\leq n},Q, P_n^n, h$ with $S_{P_n^n}=h$:
\begin{equation}\label{rev} P_j^n(dx)=(1-b_{j+1})\frac{x P_{j+1}^n(dx)}{\int y P_{j+1}^n(dy)}+b_{j+1}Q(dx), \quad \forall \,0\leq j\leq n-1.\end{equation}
Here $h,Q$ are from the general model and $P_n^n$ can be any measure in $M_1$ satisfying $S_{P_n^n}=h$. The $(b_j)_{1\leq j\leq n}$ are the first $n$ mutation probabilities in the general model. Here we use the index $j$ to indicate that we are dealing with a finite backward sequence.  

The sequence is {\it  backward} in the sense that we use $b_n$ to generate $P_{n-1}^n$ from  $P_{n}^n$, and use $b_{n-1}$ to generate $P_{n-2}^n$ from  $P_{n-1}^n$, etc. The $(b_j)$ are used backwards and the $(P_j^n)$ are generated backwards. 
The advantage to take a backward approach is that $( P_j^n)$ converges as $n$ tends to infinity, in contrast to the forward sequence. 

\begin{lemma}\label{limitgf}
In the general model, for the finite backward sequence with $P_n^n=\delta_h,$ $P_j^n$ converges in total variation to a limit, denoted by $\G_j=\G_{j,h}$, as $n$ goes to infinity with $j$ fixed, such that 
\begin{equation}\label{gf} \G_{j-1}(dx)=(1-b_{j})\frac{x \G_j(dx)}{\int y \G_j(dy)}+b_{j}Q(dx), \quad j\geq 1.\end{equation}
As a consequence, $\G_0:[0,1)^\infty\to M_1$ is a measurable function, with $\G_j=\G_0(b_{j+1},b_{j+2,\cdots})$ supported on $[0,S_Q]\cup \{h\}$ for any $j\geq0.$
\end{lemma}
\begin{remark}
We write $\G_{j,h}$ when $h$ has to be specified for clarity. Otherwise we write $\G_j$. This logic applies to other terms which will appear later. 
\end{remark}
\begin{remark}\label{conh}
Note that, by \eqref{gf}, either $\G_j(h)$'s are all zero or all strictly positive. 
\end{remark}
\begin{proof}
We prove a stronger version  below
\begin{equation}\label{preclong}\text{ For any } j, \quad P_j^n\leq_{h-}\stackrel{TV}{\longrightarrow} \G_j,  \text{ as } n\to\infty.\end{equation}
It suffices to show that 
\begin{equation}\label{pjn}P_j^n\leq_{h-}P_j^{n+1},\end{equation}
as $P_j^n$'s are all supported on $[0,S_Q]\cup\{h\}$. 

First of all, $P_n^n=\delta_h\leq_{h-} P_n^{n+1}$. Assume for some $1\leq j\leq n$, we have $P_j^n\leq_{h-} P_j^{n+1}$. By definition
\begin{equation}\label{pp+1}P_{j-1}^n(dx)=(1-b_{j})\frac{xP_{j}^n(dx)}{\int yP_{j}^n(dy)}+b_{j}Q(dx), \,\, P_{j-1}^{n+1}(dx)=(1-b_{j})\frac{x P_{j}^{n+1}(dx)}{\int yP_{j}^{n+1}(dy)}+b_{j}Q(dx).\end{equation}
Since $P_j^n\leq_{h-} P_j^{n+1}$ (hence  $P_j^{n+1} \preceq P_j^n$, see Remark \ref{meascomp}), we have $$\int yP_{j}^{n+1}(dy)\leq \int yP_{j}^n(dy)$$ and thus 
$$\frac{x}{\int yP_{j}^n(dy)}\leq \frac{x}{\int yP_{j}^{n+1}(dy)}, \quad \forall x\in [0,1].$$ 
Together with $P_j^n\leq_{h-}  P_j^{n+1}$ and (\ref{pp+1}), we get $P_{j-1}^n\leq_{h-} P_{j-1}^{n+1}$. The induction shows that 
\begin{equation}\label{prec}P_j^n\leq_{h-} P_j^{n+1} , \text{ for any } 0\leq j\leq n, n\geq  0. \end{equation}
This completes the proof.\qed
\end{proof}

The monotonicity analysis in the above proof will be used many times in this paper, as it applies to both backward and forward sequences. An immediate application is the following:
we can compare $(\G_j)$ and $(\G_j')=(\G_{j,h'})$ for $S_Q\leq h< h'\leq 1$ with the same $(b_j), Q$.

\begin{corollary}\label{kk'}
Let $(\G_j)$ and $(\G_j')$ be the above sequences. Then we have 
\begin{equation}\label{s1}\G_j'\leq_{h-}\G_j, \quad \G_j(h)\leq \G_j'(h'), \quad \forall j\geq0.\end{equation}
Moreover we have the exact equalities in the above display for any $h\in[S_Q, h']$ if and only if $\G_{0}'(h')=0$. 
In this case $(\G_j)$ and $(\G_j')$ are all supported on $[0,S_Q]$, and both equal to $(\G_{j,S_Q})$. 
\end{corollary}
\begin{proof}
Let $( P_j^n)$ be the sequence in Lemma \ref{limitgf}. Let $(P_{j,h'}^n)$ be the variant of $( P_j^n)$ with $P_n^n=\delta_{h'}.$
By following the same monotonicity analysis for proving \eqref{pjn}, we obtain 
$$P_{j,h'}^n\leq_{h-}P_j^n, \quad  P_j^n(h)\leq P_{j,h'}^n(h'), \quad\forall 0\leq j\leq n.$$
By Lemma \ref{limitgf}, $P_{j,h'}^n\stackrel{TV}{\longrightarrow}\G_j'$ and $P_{j}^n\stackrel{TV}{\longrightarrow}\G_j,$ as $n\to\infty.$
Then we obtain \eqref{s1}.  

Now let us prove the if-and-only-if statement. If $\G_0'(h')=0$, then by \eqref{s1}, $\G_0(h)=0$. Using Remark \ref{conh}, 
$\G_j'(h')=0, \G_j(h)=0$ for any $j$, and so \eqref{s1} holds with equalities. For the other direction, if $\G_0'(h')>0$, then again by Remark \ref{conh}, $\G_j'(h')>0$ for any $j$. Using \eqref{s1}, it holds that
\begin{equation}\label{'>j}\int y\G_j'(dy)>\int y\G_j(dy),\quad \forall j.\end{equation}
Similar to \eqref{pp+1}
$$\G_{j-1}(dx)=(1-b_{j})\frac{x\G_{j}(dx)}{\int y\G_{j}(dy)}+b_{j}Q(dx), \,\, \G'_{j-1}(dx)=(1-b_{j})\frac{x \G'_{j}(dx)}{\int y\G'_{j}(dy)}+b_{j}Q(dx).$$
As \eqref{'>j} entails that  $\frac{1-b_{j}}{\int y\G_{j}(dy)}>\frac{1-b_{j}}{\int y\G'_{j}(dy)}$, and using again \eqref{s1}, we obtain $\G'_{j-1}\leq_{h-}\G_{j-1}$ but they are not equal on $[0,h)$. Since they are probability measures, we have $\G_{j-1}(h)<\G_{j-1}'(h')$ for any $j$.  Then the proof is complete.

If \eqref{s1} holds with equalities, $(\G_j)=(\G_j')$ are all supported on $[0,S_Q]$. To show that they are equal to $(\G_{j,S_Q})$, we only have to take $h=S_Q$ and apply the equalities in \eqref{s1}. 
\end{proof}

 \subsubsection{The random model.}
Our goal of the paper is the random model, which is a randomised general model. Since $(\G_j)$ has parameters $(b_{j+1},b_{j+2},\cdots)$ and $Q, h$, we can define 
$$\I_j=\I_{j,h}:=\G_0(\beta_{j+1},\beta_{j+2},\cdots).$$
Therefore $\I_j$ is the quenched limit of the finite backward sequences in the random model with $P_n^n=\delta_h$. 
Thanks to Lemma \ref{limitgf}, we have the following result. 

%More generally a dynamical system is easier to handle if we take a backward point of view, see Diaconis and Freedman \cite{DF99}. 
%As a consequence of Lemma \ref{limitgf} and the fact that $(\beta_j)$ is i.i.d., we have 
\begin{corollary}\label{ier}
The sequence $(\I_j)=(\I_j)_{j\geq 0}$ is stationary ergodic and satisfies
\begin{equation}\label{if}
 \I_{j-1}(dx)=(1-\beta_{j})\frac{x \I_j(dx)}{\int y \I_j(dy)}+\beta_{j}Q(dx), \quad j\geq1.
\end{equation}
\end{corollary}
\begin{remark}
The equality \eqref{if} holds in the pointwise sense.  In other words, given any realisation of $(\beta_j)$ (or equivalently, conditioning on $(\beta_j)$), the equality holds for any $j$ as in the general model. In the sequel, when we present results regarding $(\I_j)$, then {\it conditioning on $(\beta_j)$} should be understood as {\it in the pointwise sense}. Sometimes we omit saying either of them when the context is clear. 
\end{remark}
The proof of Corollary \ref{ier} requires the following lemma which is proved by Lemma 9.5 in \cite{K97}.

%\begin{proof}
%We shall only prove the second statement. 
%Due to (\ref{gf}), $\G_{0}'(h')=0$ (resp. $>0$) is equivalent to that 
%$\G_n'(h')=0$ (resp. $>0$) for any $n\geq0.$ 

%If $\G_{0}'(h')=0$, then by (\ref{s1}), $\G_n'(h')=\G_n(h)=0$ and then $\G_n'=\G_n$. If 
%$\G_{0}'(h')>0$, again by (\ref{s1}), 
%$$\frac{x}{\int y\G_{n+1}(dy)}>\frac{x}{\int y\G_{n+1}'(dy)}, \forall x\in [0,h].$$
%We conclude by using (\ref{gf}) that $\G_n([0,h))>\G_n'([0,h))$ and so $\G_n(h)>\G_n'(h')$. 
%\end{proof}

\begin{lemma}\label{ka}
Let $(S,\SS)$ and $(S',\SS')$ be measurable spaces. Let  
$(\alpha_j)\in S^\infty$ be a stationary ergodic sequence of random variables. Let $f:S^\infty\to S'$ be a measurable function. Then $\left(f(\alpha_j,\alpha_{j+1},\cdots)\right)$ is also stationary ergodic. 
\end{lemma}

\begin{proof}[Proof of Corollary \ref{ier}]
Since $(\beta_j)$ is i.i.d., it is stationary ergodic.  As $\G_0$ is a measurable function 
from $[0,1)^\infty$ to $M_1$, we apply Lemma \ref{ka} to obtain that $(\I_j)=(\G_0(\beta_{j+1},\beta_{j+2},\cdots))$
is also stationary ergodic. The recursive equation (\ref{if}) is inherited from (\ref{gf}). \qed
\end{proof}

Since $(\I_j)$ is stationary ergodic, all $\I_j$'s have the same distribution. 
We denote by $\I:=\I_0=\I_{0,h}$ which is the weak limit appeared in Theorem \ref{main}. 
The reason to drop off the index is to make it stand out from the backward context, when it is appropriate to do so. 
The term $\I_Q$ used in Theorem \ref{critical} is in fact $\I_{0,S_Q}.$

We comment further on the importance of finite backward sequences. Let $(P_n)$ be a forward sequence and $(P_j^n)$ the finite backward sequence with $P_n^n=P_0$, both in the random model with the same $(\beta_j)$ and $Q$. Since $(\beta_j)$ is i.i.d., we have
\begin{equation}\label{for=back}(P_0, P_1,\cdots, P_n)\stackrel{d}{=}(P_n^n, P_{n-1}^n,\cdots, P_0^n).\end{equation}
So showing the weak convergence of $(P_n)_{n\geq 0}$ is equivalent to showing that of $(P_0^n)_{n\geq 0}.$ But investigating the finite backward sequences, via the general model, appears to be more convenient. In general a dynamical system is easier to handle if we take a
backward point of view, see Diaconis and Freedman \cite{DF99}. 

\subsection{Finer analysis of the finite backward sequences}\label{lfbs}
 %This idea of using randomness backwards can be traced back at least to (\cite{DF99}).
%We will work out the proof of Theorem \ref{main} in several different cases, depending on the 
%behaviour of $\I_0$. So first of all, we state a result of $\I_0$ below. 
\subsubsection{The general model.}
We consider $(P_j^n)$ with 
$P_n^n=\delta_h$, the one in Lemma \ref{limitgf}. Developing (\ref{rev}) we obtain
\begin{align}P_0^n(dx)&\label{devback1}=\left(\prod_{l=1}^n\frac{1-b_l}{\int yP_l^n(dy)}\right)x^nP_n^n(dx)+\sum_{j=0}^{n-1}\left(\prod_{l=1}^j\frac{1-b_l}{\int yP_l^n(dy)}\right)b_{j+1}m_jQ^j(dx)&\\
&\label{devback}=\left(\prod_{l=1}^n\frac{h(1-b_l)}{\int yP_l^n(dy)}\right)\delta_h(dx)+\sum_{j=0}^{n-1}\left(\prod_{l=1}^j\frac{1-b_l}{\int yP_l^n(dy)}\right)b_{j+1}m_jQ^j(dx).&\end{align}
We refer to \eqref{P'} for the expansion of the forward sequence $(P_n)$. 
\begin{proposition}\label{inv4}
Let $(P_j^n)$ be the finite backward sequence in the general model with 
$P_n^n=\delta_h$. Then for the sequence $(\G_j)$, we have 
\begin{equation}\label{bsfas}\G_0(dx){=}G_0\delta_h(dx)+\sum_{j=0}^{\infty}\prod_{l=1}^{j}\frac{(1-b_l)}{\int y\G_l(dy)}b_{j+1}m_jQ^j(dx), \end{equation}
where the second term on the right side of (\ref{devback}) converges to that of (\ref{bsfas}): 
\begin{equation}\label{pk}
\sum_{j=0}^{n-1}\left(\prod_{l=1}^j\frac{1-b_l}{\int yP_l^n(dy)}\right)b_{j+1}m_jQ^j(dx)\leq_{S_Q-}\stackrel{TV}{\longrightarrow} \sum_{j=0}^{\infty}\prod_{l=1}^{j}\frac{(1-b_l)}{\int y\G_l(dy)}b_{j+1}m_jQ^j(dx)
\end{equation}
and the term $G_0=G_{0,h}$ satisfies the following assertions: 
\begin{align}
&\label{k0pn}\prod_{l=1}^n\frac{h(1-b_l)}{\int yP_l^n(dy)}\text{ decreases in }n \text{ and converges to }G_0,&\\
&\label{k0=}G_0=1-\sum_{j=0}^{\infty}\prod_{l=1}^{j}\frac{(1-b_l)}{\int y\G_l(dy)}b_{j+1}m_j\in[0,1],&\\
&\label{k0q0}G_0=\G_0(h) \text{ , if }Q(h)=0,&\\
&\label{k0k}\int \left(\frac{y}{h}\right)^n\G_n(dy)\prod_{l=1}^n\frac{h(1-b_l)}{\int y\G_l(dy)}\text{ decreases in }n \text{ and converges to } G_0, \text{ if }G_0>0.&
\end{align}
Moreover if we define $G_j$ for $\G_j$ similarly as $G_0$ for $\G_0$, we have 
\begin{equation}\label{koj} G_{j-1}=G_j\frac{h(1-b_j)}{\int y\G_j(dy)}, \quad \forall j\geq 1.\end{equation}
As a consequence $G_j$'s are either all $0$ or all strictly positive. 
\end{proposition}

\begin{proof}
By (\ref{preclong}), $\int yP_l^n(dy)$ increases in $n$ and converges to $\int y\G_l(dy)$ as $n\to\infty$. Then using (\ref{devback}), we obtain (\ref{pk}). Integrating on both sides of \eqref{devback}, we use \eqref{pk} to deduce that
$$\prod_{l=1}^n\frac{h(1-b_l)}{\int yP_l^n(dy)}=1-\int\sum_{j=0}^{n-1}\left(\prod_{l=1}^j\frac{1-b_l}{\int yP_l^n(dy)}\right)b_{j+1}m_jQ^j(dx)$$
decreases in $n$ and converges to the limit 
$$1-\int\sum_{j=0}^{\infty}\left(\prod_{l=1}^j\frac{1-b_l}{\int y\G_l(dy)}\right)b_{j+1}m_jQ^j(dx)=:G_0.$$
So \eqref{bsfas}, \eqref{k0pn} and \eqref{k0=} are proved. 

From (\ref{bsfas}) we observe (\ref{k0q0}). To show (\ref{k0k}), 
we develop (\ref{gf}) as follows
\begin{align*}\G_0(dx)
&=\left(\prod_{l=1}^n\frac{1-b_l}{\int y\G_l(dy)}\right)x^n\G_n(dx)+\sum_{j=0}^{n-1}\left(\prod_{l=1}^j\frac{1-b_l}{\int y\G_l(dy)}\right)b_{j+1}m_jQ^j(dx)&\nonumber\\
&=\left(\int \left(\frac{y}{h}\right)^n\G_n(dy)\prod_{l=1}^n\frac{h(1-b_l)}{\int y\G_l(dy)}\right)\frac{x^n\G_n(dx)}{\int y^n\G_n(dy)}+\sum_{j=0}^{n-1}\left(\prod_{l=1}^j\frac{1-b_l}{\int y\G_l(dy)}\right)b_{j+1}m_jQ^j(dx).&
\end{align*}
Combining the above display with (\ref{bsfas}) and \eqref{pk},  we obtain (\ref{k0k}),  and also that $\frac{x^n\G_n(dx)}{\int y^n\G_n(dy)}$ converges weakly to $\delta_h$.
Finally, combining (\ref{gf}) and (\ref{bsfas}) leads to (\ref{koj}).\qed
\end{proof}
\begin{remark}\label{convergence}
The proposition implies that $(P_0^n)$ with $P_n^n=\delta_h$ in the random model converges in total variation to $\I=\I_0$, pointwise. Then by (\ref{for=back}), $(P_n)$ in the random model with $P_0=\delta_h$ converges weakly to $\I.$ Therefore Theorem \ref{main} is proved for the particular case with $P_0=\delta_h.$ As will be clear later (Section \ref{proofmain}), a complete proof has to deal with different kinds of $P_0$. The one solved here with $P_0=\delta_h$ is the simplest case. 
\end{remark}

\subsubsection{The random model.}

When carrying over the results of Proposition \ref{inv4} to the random model, we change the symbol $G$ to $I$, similar to the change from $\G$ to $\I.$  For instance, we set $I_j=G_0(\beta_{j+1}, \beta_{j+2},\cdots)$ for any $j\geq 0$. Then we have the following corollary.

\begin{corollary}\label{iergo}
The process $(I_j)=(I_j)_{j\geq 0}$ is stationary ergodic. 
Moreover $\P(\{I_j=0,\forall j\})=\P(I_0=0)\in\{0,1\}$.
\end{corollary}
\begin{remark}\label{eitheror}
If $Q(h)>0$, then it must be that $h=S_Q$ and $\I(h)=\I_0(h)>0$ a.s.. If $Q(h)=0,$ then $\I(h)=\I_0(h)=I_0.$ So applying Corollary \eqref{iergo}, either $\I(h)>0$ a.s., or  $\I(h)=0$ a.s..
\end{remark}
\begin{proof}
By Proposition \ref{inv4}, $G_0=G_0(b_1, b_2,\cdots)$ is a measurable function from 
$[0,1)^\infty$ to $[0,1]$.  As $(\beta_j)$ is i.i.d., we obtain that $(I_j)=(G_0(\beta_{j+1}, \beta_{j+2},\cdots))$ is stationary ergodic, thanks to Lemma \ref{ka}.

By (\ref{koj}), for any $k$, $\{I_k=0\}=\{I_j=0,\forall j\}$. Note that $\{I_j=0,\forall j\}$ is an invariant set in the sigma algebra generated by $(I_j)$. By ergodicity of $(I_j)$, $\P(\{I_j=0,\forall j\})=\P(I_0=0)\in\{0,1\}$. \qed
\end{proof}

The following result provides us a tool 
to know more about $\I$ and $Q$. Let $I=I_{0,h}$, and $I_Q=I_{0,S_Q}$.
To summarise, $\I, I, \I_Q, I_Q$ are identical in value to $\I_{0,h}, I_{0,h}, \I_{0,S_Q}, I_{0,S_Q}$ respectively. 

\begin{corollary}\label{pc}
The following statements about $\E\left[\ln\frac{1-\beta}{\int y\I(dy)}\right]$ hold:
\begin{enumerate}
\item[1).] $\E\left[\ln\frac{1-\beta}{\int y\I(dy)}\right]$ is well defined, and takes values in $[-\infty, -\ln \int yQ(dy)]$, and depends only on the marginal distributions of $\beta$ and $\I.$  

\item [2).] If $Q(h)=0$, then
$$ \E\left[\ln\frac{h(1-\beta)}{\int y\I(dy)}\right]\leq 0.$$

\item[3).] If $\I(h)>0$ a.s. and $Q(h)=0,$ %
then %$C_{0}\stackrel{d}{=}\nu(h)$ and 
$$\E\left[\ln\frac{h(1-\beta)}{\int y\I(dy)}\right]=0.$$

\item[4).] If $h=S_Q$ and $Q(S_Q)>0$, then 
$$\E\left[\ln \frac{S_Q(1-\beta)}{\int y\I(dy)}\right]<0\text{ and } I=0, a.s..$$ 

%\item[5).] If $h=S_Q$, then $$\E\left[\ln\frac{S_Q(1-\beta)}{\int y\nu(dy)}\right]\leq 0.$$ 

\end{enumerate}
\end{corollary}
\begin{remark}\label{e<=0}
If $h=S_Q$, we can only have $\E\left[\frac{S_Q(1-\beta)}{\int y\I(dy)}\right]=\E\left[\frac{S_Q(1-\beta)}{\int y\I_Q(dy)}\right]\leq 0.$
\end{remark}
\begin{proof}
 %By (\ref{ccon}), conditional on $(P_j)_{j\geq 0}$, 
%\begin{align*}C_{n-1}&=\lim_{k\to\infty}\prod_{j=i}^{i+k-1}\frac{(1-\beta_j)}{\int yP_{j}(dy)}\int y^{k}P_{n+k-1}(dy)&\\
%&=\lim_{k\to\infty}\prod_{j=i}^{i+k-1}\frac{h(1-\beta_j)}{\int yP_{j}(dy)}\int \frac{y^{k}}{h^{k}}P_{j+k-1}(dy)&\\
%&=\frac{h(1-\beta_n)}{\int yP_{n}(dy)}\lim_{k\to\infty}\prod_{j=j+1}^{i+k-1}\frac{(1-\beta_j)}{\int yP_{j}(dy)}\int y^{k}P_{j+k-1}(dy)&\\
%&=C_n\frac{h(1-\beta_n)}{\int yP_{n}(dy)}.&\end{align*}
%Since $(P_0)_{n\geq 1}$ is stationary, $(C_n)_{n\geq 0}$ is also stationary. 

1). By (\ref{bsfas}), $\G_0$ is a convex combination of probability measures
$\{\delta_h, Q, Q^1,Q^2, \cdots\}$. As $Q^j\preceq Q^{j+1}\preceq \delta_h$ for any $j\geq 0$,  we have, in the pointwise sense
\begin{equation}\label{3}
Q\preceq \I=\I_0\preceq \delta_h. 
\end{equation}
Then
$$\ln \int yQ(dy)\le \E\left[\ln\int y\I(dy)\right]\leq \ln h.$$
So $\E\left[\ln\int y\I(dy)\right]$ is a finite term. Consequently,
\begin{align*}
\E\left[\ln \frac{1-\beta}{\int y\I(dy)}\right]&=\E\left[\ln (1-\beta)-\ln \int y\I(dy)\right]&\\
&=\E\left[\ln (1-\beta)\right]-\E\left[\ln \int y\I(dy)\right] \in\left[-\infty, -\ln \int yQ(dy)\right].&\end{align*}
We observe that the above display depends only on the marginal distributions of $\beta$ and $\I.$

2).  Let $(P_j^n)$ be the finite backward sequence in the random model with 
$P_n^n=\delta_h$. By assumption, $Q(h)=0$. Adapting (\ref{devback}) into the random model and taking the expectation of the mass on $h$ we obtain
$$1\geq \E[P_0^{n}(h)]=\E\left[\left(\prod_{l=1}^{n}\frac{h(1-\beta_{l})}{\int yP_l^n(dy)}\right)\right]\geq \exp\left(\sum_{l=1}^{n}\E\left[\ln\frac{h(1-\beta_{l})}{\int yP_l^n(dy)}\right]\right)$$
where the second inequality is due to Jensen's inequality. By (\ref{preclong})
$$\E\left[\ln\frac{h(1-\beta_{l})}{\int yP_l^n(dy)}\right] \text{ increases in $n$ and converges to  }\E\left[\ln\frac{h(1-\beta_{l})}{\int y\I_l(dy)}\right]=\E\left[\ln\frac{h(1-\beta)}{\int y\I(dy)}\right]. $$
Combining the above two displays, it must be that $\E[\ln\frac{h(1-\beta)}{\int y\I(dy)}]\leq 0.$ %If 

3). Lemma \ref{limitgf} implies that there exists a measurable function $T:[0,1)^\infty\mapsto (0,\infty)$ such that for any $j$, $\frac{h(1-b_j)}{\int y\G_{j}(dy)}=T(b_j, b_{j+1},\cdots)$. By Lemma \ref{ka}, 
$$\left(\frac{h(1-\beta_j)}{\int y\I_{j}(dy)}\right) \text{ is stationary ergodic.}$$
By $ (\ref{k0k})$ and the fact that $\I(h)=\I_0(h)=I_0>0$ a.s.\ (because $Q(h)=0$ by assumption), 
\begin{equation}\label{1/n}\lim_{n\to\infty}(I_0)^{1/n}=\lim_{n\to\infty}\exp\left(\frac{1}{n}\ln \int\left(\frac{y}{h}\right)^n\I_n(dy)+\frac{1}{n}\sum_{l=1}^n\ln\frac{h(1-\beta_l)}{\int y\I_l(dy)}\right)\stackrel{a.s.}{=}1.\end{equation}
As $(\I_j)$ is stationary ergodic, $\int\left(\frac{y}{h}\right)^n\I_{n}(dy)\in [I_n,1]$ converges weakly to $I_0, $ which is strictly positive. Then
$$\left[\frac{1}{n}\ln I_{n}, 0\right]\ni \frac{1}{n}\ln \int \left(\frac{y}{h}\right)^n\I_n(dy)\stackrel{d}{\longrightarrow} 0, \quad n\to\infty.$$
Moreover,  since $\left(\frac{h(1-\beta_j)}{\int y\I_{j}(dy)}\right)$ is stationary ergodic, we have 
$$\frac{1}{n}\sum_{l=1}^n\ln\frac{h(1-\beta_l)}{\int y\I_l(dy)}\stackrel{a.s.}{\longrightarrow} \E\left[\ln \frac{h(1-\beta)}{\int y\I(dy)}\right], \quad n\to\infty.$$
The above three displays yield
$$1=\exp\left(\E\left[\ln \frac{h(1-\beta)}{\int y\I(dy)}\right]\right)\,\,\text{ or equivalently }\,\,\E\left[\ln \frac{h(1-\beta)}{\int y\I(dy)}\right]=0.$$

4). We show $I=I_0=0$ a.s.\ by contradiction. Adapting (\ref{devback})
in the random model
$$P_0^n(dx)=\left(\prod_{l=1}^n\frac{S_Q(1-\beta_l)}{\int yP_l^n(dy)}\right)\delta_{S_Q}(dx)+\sum_{j=0}^{n-1}\left(\prod_{l=1}^j\frac{1-\beta_l}{\int yP_l^n(dy)}\right)\beta_{j+1}m_jQ^j(dx).$$
If $I_0>0$ a.s., we consider the mass on $S_Q$ in the above display.
Note that $m_jQ^j(S_Q)=S_Q^jQ(S_Q)$. Together with (\ref{k0pn}) we obtain
$$1\geq P_0^n(S_Q)\geq Q(S_Q)\sum_{j=0}^{n-1}\left(\prod_{l=1}^j\frac{S_Q(1-\beta_l)}{\int yP_l^n(dy)}\right)\beta_{j+1}\geq Q(S_Q)\sum_{j=0}^{n-1}I_0\beta_{j+1}\stackrel{d}{\longrightarrow}\infty.$$
This is a contradiction. So $I_0=0, a.s..$ Note that by (\ref{3}), $\I(S_Q)=\I_0(S_Q)\geq Q(S_Q)>0.$ Then we get $\E\left[\ln \frac{h(1-\beta)}{\int y\I(dy)}\right]<0$ using (\ref{1/n}) and the arguments thereafter. \qed

%Consider a forward sequence $(P_n)_{n\geq 0}$ in the i-i-d model with $P_0\stackrel{d}{=}\I$.  Then $P_j\stackrel{d}{=}\I, \forall j\geq0.$ Developing (\ref{ipi}) we obtain

%\begin{equation}\label{dev}P_{j+1}(dx)=\left(\prod_{l=0}^i\frac{1-\beta_{l+1}}{\int yP_l(dy)}\right)x^{j+1}P_0(dx)+\sum_{j=1}^{j+1}\left(\prod_{l=0}^i\frac{1-\beta_{l+1}}{\int yP_l(dy)}\right)\beta_jx^{i-j+1}Q(dx), j\geq0.\end{equation}
% (\ref{pilong}), we have 

\end{proof}

\subsection{Proof of Theorem \ref{critical}}\label{proofcritical}

\begin{proof}[Proof of Theorem \ref{critical}]

The statement about $\E\left[\ln \frac{h(1-\beta)}{\int y\I_Q(dy)}\right]$ concerns just a subcase of Corollary \ref{pc}--1). So this is proved. 

If there is no condensation at $h$, then by Corollary \ref{kk'}, $\I= \I_{0,h}=\I_{0, S_Q}=\I_Q$, then of course $\I\stackrel{d}{=}\I_Q.$

The first assertion in the condensation criterion holds due to Corollary \ref{pc}--3). We consider the second one. If there is condensation at $h$, then $ \I_{0, S_Q}\neq \I_{0,h}$. By Corollary \ref{kk'}, $\I_{0,h}\leq_{S_Q-}\I_{0, S_Q}$ and $\I_{0,S_Q}(S_Q)\leq \I_{0,h}(h)$, which, together with Corollary \ref{pc}--3) entails that
$$\E\left[\ln \frac{h(1-\beta)}{\int y\I_{0,S_Q}(dy)}\right]>\E\left[\ln \frac{h(1-\beta)}{\int y\I_{0,h}(dy)}\right]=0.$$ 
The above inequality is strict because $ \I_{0, S_Q}\neq \I_{0,h}$. 

If there is no condensation at $h$, then by Corollary \ref{kk'}, $\I_{0,h}=\I_{0,S_Q}$. Using Corollary \ref{pc}--2), $$\E\left[\ln \frac{h(1-\beta)}{\int y\I_{0,S_Q}(dy)}\right]=\E\left[\ln \frac{h(1-\beta)}{\int y\I_{0,h}(dy)}\right]\leq 0.$$\qed
\end{proof}

\subsection{Some properties of invariant measures}\label{someinv}
In this section, we prove some results concerning invariant measures. But we leave the proof of Theorem \ref{com} to the end. Invariant measures will play important roles in the proof of Theorem \ref{main}. 

\begin{lemma}\label{enu}
For any invariant measure $\nu$, $\E\left[\ln\frac{1-\beta}{\int y\nu(dy)}\right]$ is well defined, and takes values in 
$[-\infty, -\ln\int yQ(dy)]$, and depends only on the marginal distributions of $\beta$ and $\nu.$ 
\end{lemma}

\begin{proof}
By the definition of invariant measure 
\begin{align*}\E\left[\int y\nu(dy)\right]&=(1-\E[\beta])\E\left[\frac{\int y^2\nu(dy)}{\int y\nu(dy)}\right]+\E[\beta]\int yQ(dy)&\\
&\geq (1-\E[\beta])\E\left[\int y\nu(dy)\right]+\E[\beta]\int yQ(dy)&\end{align*}
where the inequality is due to the fact that $\int y^2\nu(dy)\geq (\int y\nu(dy))^2.$ Then we obtain
$$\int yQ(dy)\leq \E\left[\int y\nu(dy)\right]\leq 1.$$
Proceeding similarly as in the proof of Corollary \ref{pc}--1), we conclude that this lemma holds. \qed
\end{proof}

\begin{corollary}\label{uq}
$\I_Q$ is the unique (in distribution) invariant measure supported on $[0,S_Q]$.
\end{corollary}
\begin{proof}
Let $\nu$ be any invariant measure on $[0,S_Q]$. We show that $\nu\stackrel{d}{=}\I_Q$. Note that $S_\nu=S_Q$, a.s..
Let $(P_n)$ and $(P_n')$ be two forward sequences as in Section \ref{three} with 
$$Q=Q'; \quad h=h'=S_Q; \quad P_0\stackrel{d}{=}\nu, \quad P_0'=\delta_{S_Q}$$
and $P_0$ is independent of $(\beta_n)$.  The two sequences differ only in the starting measures (satisfying $P_0\leq_{S_Q-}P_0'$), with other parameters identical. Since $\nu$ is invariant, $P_n\stackrel{d}{=}\nu$ for any $n\geq0$. 
Using the notations $\M_n,\M_n',\W_n,\W_n'$ in Section \ref{three}, and by the monotonicity analysis as in the proof of Lemma \ref{limitgf}, we obtain  in the pointwise sense,
\begin{equation}\label{compare}\int \M_n(dx)\leq \int \M_n'(dx), \quad \W_n'\leq_{S_Q} \W_n.\end{equation}

If $I_Q=0$ a.s., by \eqref{k0pn} in Proposition \ref{inv4} and \eqref{for=back}
$$\int \M_n'(dx)\stackrel{d}{\longrightarrow}I_Q\stackrel{a.s.}{=}0.$$
Remark \ref{convergence} says that $P_n'(=\W_n'+\M_n')\stackrel{d}{\to}\I_Q.$ Then 
$$\W'_n{\stackrel{d}{\longrightarrow} }\I_Q.$$
Thus applying (\ref{compare}) and $I_Q=0,$ a.s., we obtain 
$$\int \M_n(dx)\stackrel{d}{\longrightarrow}0, \quad \W_n{\stackrel{d}{\longrightarrow} }\I_Q.$$
Consequently, 
$$P_n(=\W_n+\M_n)\stackrel{d}{\longrightarrow}\I_Q.$$
Since $\nu\stackrel{d}{=}P_n$ for any $n$, we have $\nu\stackrel{d}{=}\I_Q.$

If $I_Q>0, a.s.,$ then by Corollary \ref{pc}--4), $Q(S_Q)=0$ and $\I_Q(S_Q)=I_Q>0, a.s.$. Then by Corollary \ref{pc}--3), we have 
\begin{equation}\label{qq}\E\left[\ln\frac{S_Q(1-\beta)}{\int y\I_Q(dy)}\right]=0.\end{equation}
Using again monotonicity analysis, in a pointwise sense,
\begin{equation}\label{monoana}P_n'\leq_{S_Q-}P_n, \quad P_n(S_Q)\leq P_n'(S_Q).\end{equation}
As $P_n'\stackrel{d}{\rightarrow}\I_Q$,  and $P_n\stackrel{d}{=}\nu$ for all $n$, the above display entails that,
$$\I_Q\leq_{S_Q-}^d\nu, \quad \nu(S_Q)\preceq I_Q=\I_{Q}(S_Q).$$
 Assume that $\nu$ is not equal to $\I_Q$ in distribution, then by the above display and \eqref{qq}, 
$$\E\left[\ln\frac{S_Q(1-\beta)}{\int y\nu(dy)}\right]>0.$$
The inequality implies that for $\epsilon>0$ small enough, we have
$$\E\left[\ln\frac{(S_Q-\epsilon)(1-\beta)}{\int y\nu(dy)}\right]>0.$$
As $S_\nu=S_Q$ almost surely and $P_0\stackrel{d}{=}\nu$, 
$\int (\frac{x}{S_Q-\epsilon})^nP_0(dx)\stackrel{d}{\rightarrow}\infty,$ as $n\to\infty$. 
Using again the decomposition (\ref{hg}) in Section \ref{three}, we get
\begin{align*}1&=\E\left[\int P_0(dx) \right]\geq \E\left[\int \M_n(dx)\right]&\\
&=\E\left[\exp\left(\ln\int \left(\frac{x}{S_Q-\epsilon}\right)^nP_0(dx)+\sum_{l=0}^{n-1}\ln\frac{(S_Q-\epsilon)(1-\beta_{l+1})}{\int yP_l^n(dy)}\right)\right]&\\
&\geq \E\left[\exp\left(\sum_{l=0}^{n-1}\ln\frac{(S_Q-\epsilon)(1-\beta_{l+1})}{\int yP_l^n(dy)}\right)\right]\geq \exp\left(n\E\left[\frac{(S_Q-\epsilon)(1-\beta)}{\int y\nu(dy)}\right] \right)\stackrel{n\to\infty}{\longrightarrow}\infty&\end{align*}
where the third inequality is due to Jensen's inequality. So this is a contradiction, which means that $\nu$ is equal in distribution to $\I_Q$. \qed
\end{proof}

\subsection{Proof of Theorem \ref{main}}\label{proofmain}
\noindent{{\bf Case  1. ${P_0=\delta_h}$.}}
\begin{proof}[Proof of Theorem \ref{main}]
This is shown in Remark \ref{convergence}.\qed

\end{proof}
\noindent{{\bf Case  2. ${I_{0,h}=0}$ a.s..}}
\begin{proof}[Proof of Theorem \ref{main}]
Let $(P_n)_{n\geq 0}$,  $(P'_n)_{n\geq 0}$ be two forward sequences in Section \ref{three}
with
$$Q=Q', \quad h=h',\quad  P_0'=\delta_h.$$
So the two sequences differ only in the starting measures (satisfying $P_0\leq_{h-}P_0'$), with other parameters identical. Next it suffices to follow the same procedure as in the proof of Corollary \ref{uq} for the case $I_Q=0$ a.s.. The proof is omitted. \qed\end{proof}

\noindent{{\bf Case  3. ${I_{0,h}>0 }$ a.s. and $P_0(h)>0.$}}

%\vspace{1 mm}
First of all, we restate a result from (\cite{Y15}, p.10), where only $h=1$ is considered. But it is easily generalised to any $h$. 
Recall also the distribution function $D_{u}$ for $u\in M_1$,  introduced in Section \ref{relameas}.

\begin{lemma}\label{ca}
Let $u_1,u_2\in M_1$ be any probability measures satisfying $S_{u_1}=S_{u_2}=h$ and $u_1\leq_{h-}u_2$. If for some $\epsilon>0$ there exists $a\in(0,h)$ such that $D_{u_1}(a)+\epsilon\leq D_{u_2}(a)$, then $$\int yu_1(dy)\geq c(a,\epsilon)\int yu_2(dy)$$ where $c(a,\epsilon)=\frac{1}{1-\epsilon(h-a)}>1.$
\end{lemma}

\begin{proof}[{Proof of Theorem \ref{main}}]
Let $(P_n), (P_n')$ be the two forward sequences in the proof of Case 2. Similarly as (\ref{monoana}), conditionally on $(\beta_n)$, we have
\begin{equation}\label{monoana3}P_n'\leq_{h-}P_n,\quad P_n(h)\leq  P_n'(h),\end{equation}
implying
$$\int yP_j'(dy)\geq \int yP_j(dy),\quad \forall j\geq 0.$$
For any $\epsilon>0, a\in(0,h)$, let 
$$\kappa_n:=\#\{n: D_{P'_j}(a)+\epsilon\leq D_{P_j}(a), 0\leq j\leq n\}.$$
Note that by Proposition \ref{inv4}--4), $ Q(h)=0$. So using (\ref{hg}) and (\ref{hg'}) in Section \ref{three}
$$P'_{n}(h)=\prod_{l=0}^{n-1}\frac{h(1-\beta_{l+1})}{\int yP'_l(dy)}, \quad 
P_{n}(h)=\left(\prod_{l=0}^{n-1}\frac{h(1-\beta_{l+1})}{\int yP_l(dy)}\right)P_0(h).$$
Then by Lemma \ref{ca}, we have $\prod_{l=0}^n\int yP_l'(dy)\geq c(a,\epsilon)^{\kappa_n}\prod_{l=0}^n\int yP_l(dy).$
Therefore
$$P'_{n}(h)\leq  \frac{1}{c(a,\epsilon)^{\kappa_n}}\left(\prod_{l=0}^{n-1}\frac{h(1-\beta_{l+1})}{\int yP_l(dy)}\right) = \frac{P_n(h)}{c(a,\epsilon)^{\kappa_n}P_0(h)}\leq  \frac{1}{c(a,\epsilon)^{\kappa_n}P_0(h)}.$$
But (\ref{k0pn}) of Proposition \ref{inv4} and (\ref{for=back}) entail that $P'_{n}(h)$ converges weakly to $I_{0,h}$ which is by assumption non-zero almost surely. Then $\lim_{n\to\infty}\kappa_n<\infty$ a.s.. As $a, \epsilon$ are arbitrary numbers and by Case 1 of  this theorem $P_n'\stackrel{d}{\longrightarrow}\I_{0,h}$, we use (\ref{monoana3}) to conclude that $P_n$ also  converges weakly to $\I_{0,h}.$ \qed
\end{proof}

\noindent{{\bf Case 4. ${I_{0,h}>0 }$ a.s. and $P_0(h)=0.$}}
\begin{proof}[Proof of Theorem \ref{main}]
The idea is to use a tripling argument similarly as in the proof of Theorem 5 in \cite{Y15}. 
For any $u\in M_1$ and any $a\in [0,1]$, define 
$$u^a=u_{[0,a)}+u([a,1])\delta_a, \quad a<h$$ 
where $u_{[0,a)}$ is the restriction of $u$ on $[0,a).$ 

We distinguish between $h>S_Q$ and $h=S_Q.$ For the former, let $(P_n), (P_n'), (P_n'')$ be three forward sequences in Section \ref{three} with 
$$S_Q<h''<h=h'; \quad Q'=Q,\quad  Q''=Q^{h''}=Q;\quad  P_0'=\delta_h,\quad  P_0''=P_0^{h''}.$$
So the three sequences differ in the starting measures including the largest fitness values, but with the same $Q$ and $(\beta_n)$. Since  $P_0'(h')=\delta_h(h)=1$ and $0<P_0''(h'')\leq 1$ and, we use Case 1 for $(P_n')$ and Case 2, Case 3 for $(P_n'')$ to obtain that, 
\begin{equation}\label{13}P_n'\stackrel{d}{\longrightarrow}\I_{0,h},\quad P_n''\stackrel{d}{\longrightarrow}\I_{0,h''}.\end{equation}
Applying the monotonicity analysis, the following holds in the pointwise sense: 
\begin{equation}\label{ppp'''}P'_n\leq_{h-} P_n\leq_{h''-} P''_n.\end{equation}

 Letting $h''\to h$ and using Corollary \ref{kk'}, conditionally on $(\beta_j)$, $\I_{0,h''}$ converges weakly to a limit in $M_1$, denoted by $\nu$. 
So $\nu$ is a (pointwise) weak limit of $\I_{0,h''}$ as $h''\to h.$ 
We prove next that $\nu=\I_{0,h}.$

Since $I_{0,h}>0, a.s.$ and $h>S_Q$, by Theorem \ref{critical}, 
$$\E\left[\ln\frac{h(1-\beta)}{\int y\I_{0, S_Q}(dy)}\right]>0.$$
Then for $h''$ close enough to $h$, we also have 
$$\E\left[\ln\frac{h''(1-\beta)}{\int y\I_{0, S_Q}(dy)}\right]>0.$$
The above display entails that there is condensation at $h''$, thanks to Theorem \ref{critical}.
Together with Corollary \ref{pc}--3), we have 
\begin{equation}\label{p12}\E\left[\ln\frac{h''(1-\beta)}{\int y\I_{0, h''}(dy)}\right]=0, \quad I_{0,h''}>0, \quad a.s.\end{equation}
Since $\I_{0,h''}$ is an invariant measure, the limit $\nu$ is still an invariant measure. Using \eqref{p12} and Corollary \ref{kk'}, the pointwise convergence of $\I_{0,h''}$ to $\nu$ as $h''\to h$ entails
\begin{equation}\label{nu=0}\E\left[\ln\frac{h(1-\beta)}{\int y\nu(dy)}\right]=0,\quad \nu(h)>0 \,\,\, a.s..\end{equation}
Using Corollary \ref{kk'} again, in the pointwise sense 
$$\I_{0,h}\leq_{h''-}\I_{0,h''},\quad \I_{0,h''}(h'')\leq \I_{0,h}(h), $$
implying in the pointwise sense (since $\nu$ is a pointwise weak limit of $\I_{0,h''}$),
\begin{equation}\label{inu}\I_{0,h}\leq_{h-}\nu, \quad \nu(h)\leq  I_{0,h}.\end{equation}
On the other hand, by assumption $I_{0,h}>0, a.s.$, so using Corollary \ref{pc}--3), 
$$\E\left[\ln\frac{h(1-\beta)}{\int y\I_{0,h}(dy)}\right]=0.$$
The above display with (\ref{nu=0}) and (\ref{inu}) entails that 
\begin{equation}\label{nu=h}\nu=\I_{0,h}, \quad \text{pointwise}.\end{equation}
Therefore we proved that $\I_{0,h''}$ converges pointwise to the weak limit $\I_{0,h}$ as $h''\to h.$

Now taking into account (\ref{13}), for any continuous function $f,$
\begin{align}\label{=-=}\int f(x)P_n''(dx)&\stackrel{d}{\longrightarrow}\int f(x)\I_{0,h''}(dx)&\nonumber\\
&\xrightarrow[h''\to h ]{\text{pointwise}}\int f(x)\I_{0,h}(dx)\stackrel{d}{\longleftarrow}\int f(x)P_n'(dx).&\end{align}
Note that using \eqref{ppp'''}, for any bounded continuous increasing function $g$ we have 
$$\int g(x)P_n''(dx)\preceq \int g(x)P_n(dx)\preceq\int g(x)P_n'(dx).$$
Together with \eqref{=-=}, we obtain
$$\int g(x)P_n(dx)\stackrel{d}{\longrightarrow}\int g(x)\I_{0,h}(dx).$$
Since by \eqref{ppp'''}, $P_n'\leq_{h-}P_n$ pointwise for any $n$, the above display entails that $P_n$ converges weakly to $\I_{0,h}$, the same weak limit of $(P_n')$.

If $h=S_Q$, we follow the same procedure, except that to prove (\ref{nu=h}), we require Corollary \ref{uq}. \qed
\end{proof}

\subsection{\bf Proof of Theorem \ref{com}}\label{proofcom}
Firstly we prove two lemmas. Recall the definition of $S_{u}$ for $u\in M_1$. 
\begin{lemma}\label{ucontinu}
$S_{(\cdot)}$ is a continuous (hence measurable) function on $M_1$ with the topology of the weak convergence.  
\end{lemma}

\begin{proof}
Assume that a sequence $(u_n)$ converges weakly to $u$. If $S_{u_n}$ does not converge to $S_{u}$, then  there exists a subsequence $(u_{n_k})$ such that $S_{u_{n_k}}$ converges to a limit $a$ with $a<S_{u}$ or $a>S_{u}$. 
Without loss of generality, assume $a<S_u.$ We take a positive and continuous function $f$ supported on $(\frac{a+S_{u}}{2}, S_{u}]$ and then $\int f(x)u(dx)>0.$ But $\int f(x)u_{n_k}(dx)$ converges to $0$. This is against the weak convergence. The proof is then completed. \qed
\end{proof}

The next lemma generalises Corollary \ref{uq}. 
\begin{lemma}\label{h=i}
For any invariant measure $\nu$ with $S_{\nu}=h$ a.s., we have $\nu\stackrel{d}{=}\I.$
\end{lemma}
\begin{proof}
Let $(P_n)$ be the forward sequence in the random model with $P_0\stackrel{d}{=}\nu$
and $P_0$ independent of $(\beta_n)$.  By Theorem \ref{main}, conditionally on $P_0$, $P_n$ converges in distribution to the same random measure $\I.$ Then unconditionally $P_n\stackrel{d}{=}\nu$ converges in distribution to $\I,$ implying $\nu\stackrel{d}{=}\I$.
\end{proof}

\begin{proof}[Proof of Theorem \ref{com}]
Let $\nu$ be any invariant measure. By Definition (\ref{invnu}), $S_\nu \in [S_Q, 1]$, a.s..\ By Lemma \ref{ucontinu}, $S_\nu$ is a random variable and then by Theorem 5.3 in \cite{K97}, there exists a regular conditional distribution of $\nu$ on $S_\nu$.  

Conditioning on $S_\nu$ for both sides of (\ref{invnu}), we see that $(\nu| S_\nu)$ must be an invariant measure almost surely. By Lemma \ref{h=i}, conditionally on $S_\nu$, we have $(\nu|S_\nu)\stackrel{d}{=}\I$  almost surely,  where $\I$ is the random probability measure with parameters $\mathcal L, Q, h=S_\nu$ and satisfies $\P(S_{\I}=S_\nu|S_\nu)=1$, a.s.. Then the proof is finished. \qed

%. Due to the convergences in total variation in Lemma \ref{limitgf}, $\G_0$ is more generally a measurable function $[0,1)^\infty\times[0,1]\mapsto M_1$ for parameters  $(b_j),h.$ Let $(\beta_j)$ be independent of $\nu$. 
%Then the corresponding random measure with parameters $(\beta_j), Q, h=S_\nu$, denoted by $\I_{0,S_\nu}$, is well defined as a randomised version of $\I_{0,h}$ with $h=S_\nu$ independent of $(\beta_j)$. %.  

%By 
\end{proof}

%\section{Appendix}

%\section{Acknowledgement}

%If your paper includes appendices, then precede the first of them by the command

\appendix
\section{Analysis of $\ln\frac{h(1-b)}{\int x\K_Q(dx)}$ in Kingman's model}\label{=<0}
We discuss respectively Theorem \ref{King}-1 (i.e., $\int \frac{Q(dx)}{1-x/h}\geq b^{-1}$)  and Theorem \ref{King}-2 (i.e., $\int \frac{Q(dx)}{1-x/h}< b^{-1}$). For the former, let us compute $\int x\K_Q(dx)$ first. 
\begin{align}\label{xk}
\int x\K_Q(dx)&=\int\frac{b\theta_bxQ(dx)}{\theta_b-(1-b)x}&\nonumber\\
&=\int\frac{b\theta_b(x-\theta_b/(1-b))Q(dx)+b\theta_b^2/(1-b)Q(dx)}{\theta_b-(1-b)x}&\nonumber\\
&=\frac{b\theta_b}{1-b}+\frac{b\theta_b^2}{1-b}\int\frac{Q(dx)}{\theta_b-(1-b)x}&\nonumber\\
&=\theta_b&
\end{align}
where the last equality is due to the fact that $\theta_b$ is the solution of the equation \eqref{sb}. The equation \eqref{sb} also  implies 
$$\int\frac{\theta_bQ(dx)}{\theta_b-(1-b)x}=\int\frac{Q(dx)}{1-(1-b)x/\theta_b}=b^{-1}.$$ 
Recall $\int \frac{Q(dx)}{1-x/h}\geq b^{-1}.$ Then the above display entails that 
$$\frac{1}{h}\geq \frac{1-b}{\theta_b}.$$
Taking into account \eqref{xk}, we arrive at 
$$\frac{h(1-b)}{\int x\K_Q(dx)}\leq 1, \text{ or equivalently, } \ln \frac{h(1-b)}{\int x\K_Q(dx)}\leq 0.$$
The equality holds if and only if $\int \frac{Q(dx)}{1-x/h}=b^{-1}.$

For Theorem \ref{King}-2, we have 
 \begin{align*}
\int x\K_Q(dx)&=\int\frac{bxQ(dx)}{1-x/S_Q}+\left(1-\int\frac{bQ(dx)}{1-x/S_Q}\right)S_Q&\\
&=S_Q+\int\frac{b(x-S_Q)Q(dx)}{1-x/S_Q}&\\
&=(1-b)S_Q.&
\end{align*}
Then we obtain
$$\ln \frac{h(1-b)}{\int x\K_Q(dx)}=\ln\frac{h}{S_Q}\geq 0$$
where the equality holds if and only if $h=S_Q$. 

In conclusion, if $h>S_Q,$ $\ln \frac{h(1-b)}{\int x\K_Q(dx)}\leq 0$ is equivalent to $\int \frac{Q(dx)}{1-x/h}\geq b^{-1}$ (non-condensation case), and $\ln \frac{(1-b)S_Q}{\int x\K_Q(dx)}>0$ is equivalent to $\int \frac{Q(dx)}{1-x/h}<b^{-1}$ (condensation case).

If $h=S_Q$, $\ln \frac{h(1-b)}{\int x\K_Q(dx)}= 0$ is equivalent to either $\int \frac{Q(dx)}{1-x/S_Q}=b^{-1}$ (non-condensation case) or $\int \frac{Q(dx)}{1-x/S_Q}<b^{-1}$ (condensation case), and $\ln \frac{S_Q(1-b)}{\int x\K_Q(dx)}<0$ is equivalent to $\int \frac{Q(dx)}{1-x/S_Q}>b^{-1}$ (non-condensation case). The case $\ln \frac{S_Q(1-b)}{\int x\K_Q(dx)}>0$ does not exist, which is in line with Remark \ref{e<=0}.  

Therefore, if $h=S_Q$, knowing only $\ln \frac{h(1-b)}{\int x\K_Q(dx)}= 0$ cannot tell whether the condensation occurs or not.

%and then carry on using the \section and \subsection commands, as above.

%\section{The first appendix}

%If you include EPS (encapsulated postscript) figures in your paper,
%then please use the following commands:
%\begin{figure}
%\begin{center}
%\includegraphics{.eps}
%\caption{Caption text.}\label{}
%\end{center}
%\end{figure}

\section*{Acknowledgements} The author thanks Takis Konstantopoulos, G\"otz Kersting and Pascal Grange for discussions. The author thanks the anonymous reviewers for their comments which greatly improved the presentation of the paper, and for suggesting the intuition for Theorem \ref{critical}.  The author acknowledges the support of the National Natural Science Foundation of China (Youth Programme, Grant:
11801458).

% Place the text of your acknowledgements after the \acks command.
% \acks generates the heading "Acknowledgements".
% If you wish to make only one acknowledgement, use \ack.
% \ack generates the heading "Acknowledgement".

% Reference list
%
% References should be in the following form (or the BibTeX file
% apt.bst should be used):
%
% For a journal:
% Surname, Initial (year). Title of paper. {\em Journal title}
% {\bf Vol,} page--range.
%
% For a book:
% Surname, Initial (year). {\em Book title}. Publisher, Address.
%
% Note the following example of a reference list.

\end{document}